 \documentclass[11pt]{amsart}
\usepackage{geometry}   
\usepackage{setspace}
\usepackage{graphicx}
\usepackage{amssymb, amsthm}
\usepackage{epstopdf}
\usepackage{pinlabel}
\usepackage{color}
\usepackage{hyperref}
\usepackage{tikz}
\usetikzlibrary{calc,decorations.markings}
 \newtheorem{theorem}{Theorem}[section]
 \newtheorem*{maintheorem}{Main Theorem}
    
      \newtheorem*{cor}{Corollary}
   \newtheorem{lemma}[theorem]{Lemma}
    \newtheorem{proposition}[theorem]{Proposition}  
     \newtheorem{conjecture}[theorem]{Conjecture}

    \theoremstyle{definition}
\newtheorem{definition}[theorem]{Definition}

\newtheorem{example}[theorem]{Example}

  \newcommand{\Z}{\ensuremath{{\mathbb{Z}}}}
\newcommand{\R}{\ensuremath{{\mathbb{R}}}}

\definecolor{dred}{rgb}{.5,0,0} 
\definecolor{dgreen}{rgb}{0,.5,0} 
\definecolor{blue}{rgb}{0,0,0.5} 
\definecolor{black}{rgb}{0,0,0} 
\definecolor{vdgreen}{rgb}{0,.3,0} 
\definecolor{vdred}{rgb}{.3,0,0} 

\newcommand{\mbX}{\ensuremath{{\partial_* X}}}
\newcommand{\mbdX}{\ensuremath{{\partial_*^D X}}}
\newcommand{\mbY}{\ensuremath{{\partial_* Y}}}

\newcommand{\mbeY}{\ensuremath{{\partial_*^{D'} Y}}}
\newcommand{\mbh}{\ensuremath{{\partial_*h}}}

\title{A rank-one CAT(0) group is determined by its Morse boundary}
\author{Ruth Charney and Devin Murray}

\thanks {Both authors were partially supported by NSF grant DMS-1607616}

\begin{document}

\begin{abstract} The Morse boundary of a proper geodesic metric space is designed to encode hypberbolic-like behavior in the space.  A key property of this boundary is that a quasi-isometry between two such spaces induces a homeomorphism on their Morse boundaries.  In this paper we investigate when the converse holds.  We prove that for cocompact CAT(0) spaces, a homeomorphism of Morse boundaries is induced by a quasi-isometry if and only if the homeomorphism is quasi-mobius and 2-stable.
\end{abstract}

\maketitle

\section{Introduction}

Boundaries of hyperbolic spaces have played a major role in the study of hyperbolic geometry and hyperbolic groups.  In particular, they provide a fundamental tool for studying the dynamics of isometries and rigidity properties of hyperbolic groups.

The effectiveness of this tool depends on a few key properties.  The first, is quasi-isometry invariance:  a quasi-isometry between two hyperbolic metric spaces induces a homeomorphism on their boundaries.  
In particular, this allows us to talk about the boundary of a hyperbolic group.  
Moreover, these homeomorphisms satisfy some particularly nice properties; they are quasi-mobius and quasi-conformal. Quasi-mobius is a condition that bounds the distortion of cross-ratios while quasi-conformal bounds the distortion of metric spheres.  These conditions have been studied in a variety of contexts by Otal, Pansu, Tukia, and Vaisala, \cite{O, Pan,T86,TV82, TV84, V} among others.  One of the most general theorems can be found in a 1996 paper of Paulin \cite{Pau} where he proves that if  $f : \partial X \to \partial Y$  is a homeomorphism between the boundaries of two proper, cocompact hyperbolic spaces, then the following are equivalent
\begin{enumerate}
\item $f$ is induced by a quasi-isometry $h : X \to Y$,
\item $f$ is quasi-mobius,
\item $f$ is quasi-conformal.
\end{enumerate}
We remark that Paulin's definition of quasi-conformal is different from the one used by Tukia and others. In this paper, we will focus on the qusi-mobius condition.

Boundaries can be defined for a variety of other spaces.  In particular, one can define a  boundary for any CAT(0) space.  Unfortunately, many of the nice properties of hyperbolic boundaries fail in this context.  First, quasi-isometries of CAT(0) spaces do not, in general, induce homeomorphisms on their boundaries. A well-known example of Croke and Kleiner \cite{CK} exhibits a group acting geometrically on two CAT(0) spaces with non-homeomorphic boundaries.  The missing property that leads to the failure of quasi-isometry invariance, is that in hyperbolic spaces, quasi-geodesics stay bounded distance from geodesics (with the bound depending only on the quasi-constants) while in CAT(0) spaces, this need not hold.  This property is known as the Morse property.  

In \cite{CS} the first author and H.~Sultan introduced a new type of boundary for CAT(0) spaces by restricting to only those geodesic rays satisfying the Morse property.  For CAT(0) spaces, the Morse property is equivalent to the contracting property (see section \ref{prelim} for definitions) and the authors originally called their boundary the ``contracting boundary".  Subsequently, their construction was generalized to arbitrary proper geodesic metric spaces by M.~Cordes \cite{Co} using the Morse property.  These boundaries have thus come to be known as Morse boundaries.  We denote the Morse boundary of $X$ by $\mbX$.  The key property of this boundary is quasi-isometry invariance; a quasi-isometry between two proper geodesic metric spaces induces a homeomorphism on their Morse boundaries \cite{CS, Co}.  Thus the Morse boundary is well-defined for \emph{any} finitely generated group (though it may be empty if the group has no Morse geodesics).  For more about Morse boundaries of general groups, see Cordes' survey paper \cite{Co17}.
  
In this paper, we will restrict our attention to Morse boundaries of proper, CAT(0) spaces.  In \cite{Mu}, the second author proves that in this context, Morse boundaries have a variety of other properties analogous to hyperbolic boundaries, properties that will play a useful role in the proofs below. 
In the current paper, we prove the following analogue of Paulin's theorem.

\begin{maintheorem} Let $X$ and $Y$ be proper, cocompact CAT(0) spaces and assume that $\mbX$ contains at least 3 points.  Then a homeomorphism $f : \mbX \to \mbY$ is induced by a quasi-isometry  
$h : X \to Y$ if and only if $f$ is 2-stable and quasi-mobius.  
\end{maintheorem}
We refer the reader to Section \ref{2-stable} for the definitions of quasi-mobius and 2-stable.

In particular, this theorem applies to CAT(0) groups.  In \cite{Mu}, building on work of Ballman and Buyalo \cite{BB}, the second author showed that if $G$ acts geometrically on a CAT(0) space $X$, then $\mbX$ contains at least 3 points if and only if $G$ is rank one and not virtually cyclic. Thus the theorem can be restated for CAT(0) groups as follows.
\begin{cor}  Let $G$ and $H$ be rank one CAT(0) groups.  Then $G$ is quasi-isometric to $H$ if and only if there exists a homeomorphism $f : \partial_*G \to \partial_*H$ which is quasi-mobius and 2-stable.
\end{cor}

One might ask if there is also an equivalent quasi-conformality condition as in Paulin's theorem.  In general, however, the Morse boundary is neither metrizable nor compact, so it is not even clear what quasi-conformal should mean in this context.  However, a recent paper of Cashen and Mackey \cite{CaMa} introduces a metrizable topology on the Morse boundary which could be potentially be used to define quasi-conformal.  
It would be interesting to know whether a full analogue of Paulin's theorem holds for this modified Morse boundary.  Another natural question is whether the main theorem holds for Morse boundaries of more general geodesic metric spaces.  

The first author would like to thank the Mathematical Sciences Research Institute and the Isaac Newton Institute for Mathematical Sciences for their support during the writing of this paper.

\section{Preliminaries}\label{prelim}

\subsection{Contracting geodesics}  In this section we review some basic facts about contracting geodesics and the definition of the Morse boundary.  The reader is referred to \cite{BF, CS, Su} for details. 

We assume throughout that $X$ is a proper, CAT(0) space.  Let $\partial X$ denote the visual boundary of of $X$, that is, 
$$\partial X = \{ \alpha \mid \textrm{$\alpha : [0,\infty) \to X$ is a geodesic ray}\} / \sim$$
where two rays are equivalent if they have bounded Hausdorff distance.  The topology on $\partial X$ is given by the neighborhood basis consisting of sets of rays which stay $\epsilon$-close for distance $R$.

For a geodesic $\alpha$ in $X$ and a set $Y \subset X$, denote by $\pi_{\alpha}(Y)$ the image of the nearest point projection of $Y$ on $\alpha$.  

\begin{definition} A (finite or infinite) geodesic $\alpha$ in $X$ is \emph{$D$-contracting} if for every metric ball $B$ that does not intersect $\alpha$, the projection $\pi_{\alpha}(Y)$ has diameter at most $D$.  Or equivalently, if for any two points $x,y \in X$ with $d(x, y) < d(x,\alpha)$, the distance between $\pi_{\alpha}(x)$ and $\pi_{\alpha}(y)$ is at most $D$.  
\end{definition}

As noted above, in a CAT(0) space, the contracting property is equivalent to the Morse property.
(In the following, $\R^+$ denotes the non-negative real numbers.)

\begin{definition}  A geodesic $\alpha$ in $X$ is \emph{Morse} if there exists a function $N: \R^+ \times \R^+ \to \R^+$ such that any $(\lambda, \epsilon)$-quasi-geodesic with endpoints on $\alpha$,  lies in the $N(\lambda, \epsilon)$-neighborhood of $\alpha$.  The function $N$ is called a \emph{Morse gauge} for $\alpha$. 
 \end{definition}
 
For our purposes, the contracting property is more convenient since we will frequently be concerned with projections of sets onto geodesics.  
Contracting geodesics in a CAT(0) space satisfy a number of nice properties which we now recall briefly.
\begin{enumerate}

\item {\bf Slim Triangle Property:}   There exists $\delta_D$ depending only on $D$ such if $\alpha$ is $D$-contracting, $x \in X$, and $p=\pi_\alpha(x)$ is the projection of $x$ on $\alpha$, then for any point $y \in \alpha$, the geodesic from $x$ to $y$ lies in the $\delta_D$-neighborhood of $[x,p] \cup [p,y]$.

\item  {\bf Equivalent Geodesics:}  There exist $D'$ such that if $\alpha$ and $\beta$ are bi-infinite geodesics with the same endpoints in $\partial X$ and $\alpha$ is $D$-contracting,  then $\beta$ is $D'$ contracting and the Hausdorff distance between them is at most $2\delta_D$.

\item  {\bf Contracting Triangles:} Given $D$, there exists $D'$ such that if two sides of a triangle with endpoints in $X \cup \partial X$ are $D$-contracting, then the third side is $D'$-contracting.

\item  {\bf Bounded Geodesic Image Property:}  Given $D$, there exist a constant $B_D$ such that for any geodesics $\alpha, \beta$ with $\alpha$ $D$-contracting, either the projection of $\beta$ on $\alpha$ has diameter greater than $B_D$, or $\beta$ contains a point $z$ with  $d(z,\alpha) < B_D$.
\end{enumerate}

 \subsection{The Morse boundary}
Fix a basepoint $x_0 \in X$.  Denote by  $\mbdX$ the subspace of the visual boundary consisting of $D$-contracting rays based at $x_0$.
The Morse (or contracting) boundary of $X$ is defined as the direct limit $\mbX = \lim_{D \to \infty} \mbdX$. 
It is shown in \cite{CS} that this topology is independent of choice of basepoint.
While the Morse boundary is set-theoretically contained in the visual boundary, the direct limit topology is generally much finer than the subspace topology.  In fact the two topologies agree if and only if $X$ is hyperbolic.  

As in \cite{Mu}, we can also define a topology on $X \cup \mbX$ as follows. Let $\overline X= X \cup \partial X$ viewed as the set of generalized geodesic rays based at $x_0$.  (A generalized ray may either go to infinity or stop at a point in the interior.)  There is a standard topology on $\overline X$, which restricts to the visual topology on $\partial X$.  It has a neighborhood basis consisting of sets of rays that stay $\epsilon$-close to a given ray for distance $R$.  
Let $\overline X_*^D$ be the set of generalized rays in $\overline X$ that are $D$-contracting.  Put the subspace topology on $\overline X_*^D$ and define the topology on $\overline X_* = X \cup \mbX$ to be the direct limit topology.  This topology is independent of choice of basepoint and restricts to the given topologies on $X$ and $\mbX$.

\section{Homeomorphisms induced by quasi-isometries}

In \cite{CS}, it was shown that a quasi-isometry $h:X \to Y$ between proper CAT(0) spaces induces a homeomorphism $\mbh$ on the contracting boundaries.  In this section we will show that these homeomorphisms satisfy some additional properties.  \emph{We assume throughout that $\mbX$ contains at least three points.}

\subsection{Two-stable maps}\label{2-stable}
Recall that $\mbdX$ was defined in terms of a fixed basepoint $x_0$.  
Changing the basepoint changes the contracting constant associated to a point on the boundary.
In this paper we are concerned with bi-infinite $D$-contracting geodesics.  
Let $\alpha$ be a bi-infinte, $D$-contracting geodesic in $X$.
While its endpoints $\alpha^+$ and $\alpha ^-$ are in $\mbX$, the rays from $x_0$ to these points may have much larger contracting constants, so $\alpha^+$ and $\alpha ^-$ need not lie in $\mbdX$.  

For two points $a,b \in \overline X_*$ let $(a,b)$ denote the set of geodesics from $a$ to $b$. 
Denote by $\mbX^{(n,D)}$, $n$-tuples of distinct points
$(a_1, a_2, \dots a_n)$ in $\mbX$ such that every bi-infinite geodesic in $(a_i,a_j)$ is $D$-contracting. Since $\mbX$ has at least 3 points and the bi-infinite geodesic between any two of these is contracting,  $\mbX^{(3,D)}$ is non-empty for $D$ sufficiently large.  We will refer to 3-tuples $(a,b,c) \in \mbX^{(3,D)}$ as $D$-triangles.

\begin{definition}  Let $X$ and $Y$ be proper, cocompact CAT(0) spaces. A map $f : \mbX \to \mbY$ is \emph{two-stable} if for every $D \geq 0$, there exists $D' \geq 0$ such that $f$ maps $\mbX^{(2,D)}$ into $\mbY^{(2,D')}$.  Note that it follows that $f$ maps $\mbX^{(n,D)}$ into $\mbY^{(n,D')}$ for all $n \geq 2$.  
\end{definition}

Now suppose $f : \mbX \to \mbY$  is a homeomorphism. Since a closed set in $\mbX$ is compact if and only if it is contained in $\mbdX$, for some $D$,  it must be the case that for each $D$, there exists a $D'$ such that $f$ maps  $\mbdX$ into $\mbeY$.  On the other hand, this is does not guarantee that $f$ is 2-stable. 

\begin{example} \label{not 2-stable} Let $X$ be the Euclidean plane $\R^2$ with a ray $r_{m,n}$ attached at each lattice point $(m,n) \in \Z^2 \subset \R^2$. View the plane as horizontal and the attached rays as vertical.   It is easy to see that the contracting boundary is the discrete set of the vertical rays. For a bi-infinite geodesic between two such boundary points $r_{m,n}$ and $r_{s,t}$,  the optimal contracting constant for the bi-infinite geodesic connecting them is given by the distance in the plane from $(m,n)$ to $(s,t)$.

Consider the homeomorphism $f: \mbX \to \mbX$ which interchanges $r_{n,0}$ and $r_{-n,0}$  and leaves all other points on the boundary fixed.  Let $\alpha_n$ be the bi-infinite geodesic from $r_{n,0}$ to 
$r_{n,1}$. Then for all $n$, $\alpha_n$ is 1-contracting, whereas after applying $f$,  the bi-infinite geodesic between the resulting points, $r_{-n,0}$ and $r_{n,1}$ is worse than $2n$-contracting. Thus, $f$ is not 2-stable.

One can promote this example to a space  with a cocompact group action.  Namely, let $X$ be the universal cover of a torus wedge a circle, $T^2 \vee S^1$.  View the flats in $X$ as horizontal and the edges covering the circle as vertical.  Choose a base flat $F$ and identify it with $\R^2$.  Let $e(n,m)$ denote the upward edge attached at $(n,m) \in F$.  Define  $f: \mbX \to \mbX$ by interchanging any ray from the origin passing through   $e(n,0)$ with the corresponding ray passing through $e(-n,0)$, and leaving the rest of the boundary fixed.  This again defines a homeomorphism on the boundary which, by the same argument as above, is not 2-stable. 
\end{example}

\subsection{Cross-ratios}\label{cross-ratios}

We begin by reviewing Paulin's definition of the cross-ratio.
For four points $a,b,c,d$ in a $\delta$-hyperbolic space $X$, Paulin defines the cross-ratio to be 
$[a,b,c,d]=\frac{1}{2} (d(a,d) + d(b,c) - d(a,b) - d(c,d))$.
He then extends this definition to $\partial X$ by taking limits over sequences of points approaching the boundary.  We will use a slightly different definition of the cross-ratio motivated by the following observation. 
Let $p=\pi_{(a,c)}(b)$  and $q=\pi_{(a,c)}(d)$ be the projections of $b$ and $d$ on $(a,c)$, as in Figure \ref{fig:project}.
Using the thin triangle property, it is easy to see that (the absolute value of) Paulin's cross-ratio is approximately equal to $d(p,q)$, they differ by at most $4\delta$.  

\begin{figure}[h]

\begin{tikzpicture}[thick]

\newlength\mylen

\tikzset{
bicolor/.style n args={3}{
  decoration={
    markings,
    mark=at position 0.5 with {
      \node[draw=none,inner sep=0pt,fill=none,text width=0pt,minimum size=0pt] {\global\setlength\mylen{\pgfdecoratedpathlength}};
    },
  },
  draw=#1,
  dash pattern=on #3\mylen off 1\mylen,
  preaction={decorate},
  postaction={
    draw=#2,
    dash pattern=on (1-#3)*\mylen off (#3)*\mylen,dash phase=(1-#3)*\mylen
  },
  }
}

\tikzset{
tricolor/.style n args={5}{
  decoration={
    markings,
    mark=at position 0.5 with {
      \node[draw=none,inner sep=0pt,fill=none,text width=0pt,minimum size=0pt] {\global\setlength\mylen{\pgfdecoratedpathlength}};
    },
  },
  draw=#1,
  dash pattern=on #4\mylen off \mylen,
  preaction={decorate},
  postaction={
    draw=#2,
    dash pattern=on (#5-#4)*\mylen off \mylen,dash phase=(-#4)*\mylen
  },
  postaction={
    draw=#3,
    dash pattern=on (1-#5)*\mylen off \mylen,dash phase=(-#5)*\mylen
    },
  }
}

	\node[scale=0.4,circle,fill](a) at (0,0) {};
	\node at (a) [anchor=north east]{$a$};
	\node[scale=0.4,circle,fill](b) at (2,5) {};
	\node at (b) [anchor=south east]{$b$};
	\node[scale=0.4,circle,fill](c) at (12,3) {};
	\node at (c) [anchor=south west]{$c$};
	\node[scale=0.4,circle,fill](d) at (11,-3) {};
	\node at (d) [anchor=north]{$d$};
	\node (up) at (0,0.2) {};
	\node (down) at (0,-0.2) {};
	\node (right) at (0.2,0) {};
	\node (left) at (-0.2,0) {};
	\draw[-] (a) -- (c) node[pos=0.25,fill,scale=0.4,circle] (p){} node[pos = 0.25, anchor=south east]{$p$} node[pos=0.65,fill,scale=0.4,circle] (q){} node[pos=0.65,anchor = north west]{$q$};
	\draw[bicolor={cyan}{red}{0.45}] (a)  ..controls ($(p)+1.5*(up)+(left)$) .. (b) node[pos=0.23,anchor = south east]{$-$} node[pos=0.77,anchor = east]{$-$};
	\draw[tricolor={red}{black}{cyan}{0.29}{0.65}] (b) ..controls ($(p)+(up)$) .. (c) node[pos=0.2,anchor = west]{$+$} node[pos=0.92,anchor = south]{$+$};
	\draw[tricolor={cyan}{black}{red}{0.25}{0.58}] (a) ..controls ($(q)+2*(down)+(right)$) .. (d) node[pos=0.08,anchor = north]{$+$} node[pos=0.75,anchor = north east]{$+$};
	\draw[bicolor={red}{cyan}{0.58}] (d) ..controls ($(q)+(down)$) .. (c) node[pos=0.22,anchor = west]{$-$} node[pos=0.82,anchor = north]{$-$};

\end{tikzpicture}
\caption{Projection distance is coarsely equal to Paulin's cross-ratio}
\label{fig:project}
\end{figure}
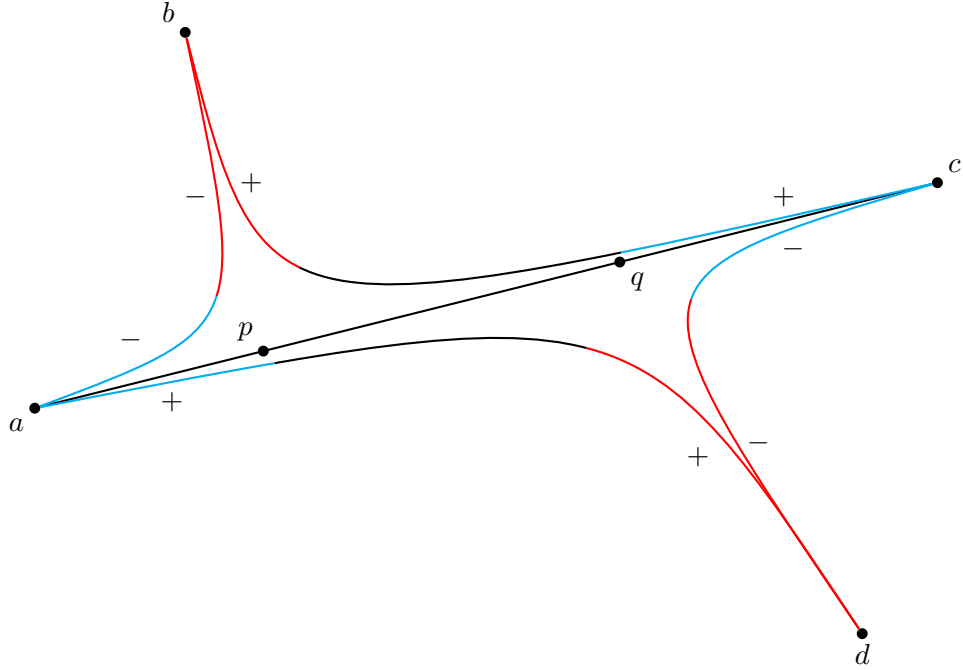

The notion of projection extends to the boundary and gives rise to a definition of the cross-ratio in
$\mbX$ that is more intrinsic and generally easier to work with.  
Following \cite{Mu}, one can define the projection of a point $b \in \mbX$ onto a $D$-contracting geodesic $\alpha$ as follows.  For any ray $\beta$ representing $b$, consider the points on $\alpha$ which are limit points of sequences $\pi_\alpha(\beta(t_i)$ with $t_i \to \infty$.  The set 
$P_\alpha(b) \subset \alpha$ of all such limit points has diameter bounded by $B_D$, where $B_D$ is the constant defined in the Bounded Geodesic Image Property.  The projection, $\pi_\alpha(b)$, is then defined as the barycenter of $P_\alpha(b)$. 

Using this projection, we can now expand the slim triangle property to include ideal triangles (at the expense of increasing the constant $\delta_D$ to take into account the diameter of the projection sets).  

\begin{lemma}[\bf{Slim Triangle Property}]\label{slim} Let $a,b,c$ be distinct points in $X \cup \mbX$ and suppose $\alpha \in (a,c)$ is $D$-contacting.  Let $p=\pi_\alpha(b)$.  Then there exists  $\delta_D$, depending only on $D$, such that any geodesic $\beta \in (a,b)$ lies in the $\delta_D$-neighborhood of $(a,p) \cup (p,b)$.
\end{lemma}

It therefore makes sense to define the cross-ratio on $\mbX$ as follows.

\begin{definition}
The \emph{cross-ratio} of a four-tuple $(a,b,c,d) \in \mbX^{(4)}$ is defined to be
$$ [a,b,c,d] = \pm \sup_{\alpha \in (a,c)} d(\pi_{\alpha}(b),\pi_{\alpha}(d))$$
where the sign is positive if the orientation of the geodesic $(\pi_{\alpha}(b),\pi_{\alpha}(d))$ agrees with that of $(a,c)$ and is negative otherwise.
\end{definition}

For the quasi-mobius property, we will want to bound the absolute value of the cross-ratio.  The following lemma shows that for this, we need only work with a specific choice of geodesic $\alpha$.

\begin{lemma}\label{supremum} If  $(a,b,c,d) \in \mbX^{(4,D)}$, then  for any choice of geodesic $\alpha \in (a,c)$,
$$ |[a,b,c,d]| -6\delta_D  \leq  d(\pi_{\alpha}(b),\pi_{\alpha}(d)) \leq |[a,b,c,d]|.$$
\end{lemma}

\begin{proof} Any two geodesics $\alpha, \alpha' \in (a,c)$ have Hausdorff distance at most $2\delta_D$.
An easy exercise then shows that the projection of any point $x \in X$ onto $\alpha$ lies within $6\delta_D$ of its projection onto $\alpha'$, so the same holds for projections of boundary points on $\alpha$ and $\alpha'$. 
\end{proof}

\begin{definition}  Let $X$ and $Y$ be CAT(0) spaces.  A homeomorphism $f: \mbX \to \mbY$ is  
\emph{$D$-quasi-mobius}  if there exists a continuous map 
 $\psi_D : [0,\infty) \to [0,\infty)$  such that for all 4-tuples $(\alpha,\alpha',\beta,\beta')$ in $\mbX^{(4,D)}$, 
$$ |[f(\alpha),f(\alpha'),f(\beta),f(\beta')]| \leq \psi_D(|[\alpha,\alpha',\beta,\beta']|). $$
We say that $f$ is \emph{quasi-mobius} if $f$ and $f^{-1}$ are both $D$-quasi-mobius for every $D$.   
\end{definition} 

We remark that one can always choose the functions $\psi_D$ to be non-decreasing.

For a point $(a,b,c) \in \mbX^{(3,D)}$, let $T(a,b,c)$ denote an ideal triangle with vertices $a,b,c$. 
That is, we specify a choice of bi-infinite geodesics as edges.  

\begin{lemma}\label{centers}   Let $(a,b,c) \in \mbX^{(3,D)}$, and set
$$E_K(a,b,c)= \{ x \in X \mid \textrm{ $x$ lies within $K$ of all three sides of some $T(a,b,c)$}\}.$$
Then for any $K \geq B_D + \delta_D$, the following hold.
\begin{enumerate}
\item For any $\alpha \in (a,c)$, $\pi_\alpha(b)$ lies in $E_K(a,b,c)$
\item $E_K(a,b,c)$ has bounded diameter with the bound depending only on $D$ and $K$.
In particular, it has a well-defined barycenter, $\pi_X(a,b,c)$
\item There is a constant $C$ depending only on $D$ and $K$, such that for any $(a,b,c,d) \in \mbX^{(4,D)}$,  $|[a,b,c,d]|$ differs from $d(\pi_X(a,b,c),\pi_X(a,c,d))$ by at most $C$.
\end{enumerate}
\end{lemma}

\begin{proof}  	
(1) Say $\alpha$, $\beta$, and $\gamma$ are the three sides of a triangle $T(a,b,c)$ with $\alpha$ the side from $a$ to $c$.  Then the slim triangle property guarantees that every point on the projection $\pi_\alpha(\beta)$ lies within $\delta_D$ of $\beta$ and similarly, every point on the projection $\pi_\alpha(\gamma)$ lies within $\delta_D$ of $\gamma$. The limit points of these projections both lie in the projection set $P_\alpha(b)$, which has diameter at most $B_D$, so they lie within $B_D$ of $\pi_\alpha(b)$ = the barycenter of $P_\alpha(b)$.  It follows that $\pi_\alpha(b)$ lies within $K$ of both $\beta$ and $\gamma$.  
		
(2) To show that $E_K(a,b,c)$ has bounded diameter, fix an ideal triangle $T=T(a,b,c)$.
Any other triangle $T'=T'(a,b,c)$ lies in the $2\delta_D$-neighborhood of $T$, so for
any point $x \in E_K(a,b,c)$, $x$ lies within $K'=K+2\delta_D$ of all three sides of $T$.  
Projecting $x$ onto the sides of $T$ gives three points on $T$ that all lie within $2K'$ of each other.  By standard arguments, this subset of $T$ has bounded diameter with the bound depending only on $K'$ and $D$.
 
(3) This follows immediately from parts (1) and (2), together with Lemma \ref{supremum}.   
\end{proof}

\begin{theorem}\label{forward} Let $h: X \to Y$ be a $(\lambda,\epsilon)$-quasi-isometry between proper CAT(0) spaces.  Then the induced map $\mbh : \mbX \to \mbY$ is a 2-stable, quasi-mobius homeomorphism.  Moreover, the functions $\psi_D$ in the definition of quasi-mobius can all be taken to be linear with multiplicative constant $\lambda$.
 \end{theorem}

\begin{proof}  The fact that $\mbh$ is a homeomorphism was proved by the first author and H.~Sultan in \cite{CS}.  In particular, using the equivalence of the contracting and Morse properties for CAT(0) geodesics, they showed that for each $D$ there exists a $D'$ such that the image of a $D$-contracting ray under the quasi-isometry $h$ can be ``straightened"  to a $D'$-contracting ray in $Y$.  The same proof applies to bi-infinite geodesics to show that $\mbh$ is 2-stable.  

Suppose $h$ is a $(\lambda, \epsilon)$-quasi-isometry.  To prove that $\mbh$ is quasi-mobius, first consider a triple $(a,b,c) \in \mbX^{(3,D)}$ and let $(a',b',c') \in \mbY^{(3,D')}$  its image under $\mbh$.  Let  $T=T(a,b,c)$ be a representative triangle.  Applying $h$ to $T$ gives a quasi-triangle (a triangle whose sides are quasi-geodesic) in $Y$.  The sides can be straightened to geodesics to obtain a triangle $T'= T'(a',b',c')$.  The Morse property guarantees that there exists an $N$, depending only on $D', \lambda, \epsilon$ such that $T'$ lies in the $N$-neighborhood of $h(T)$.  

Now set $K=B_D+ \delta_D$ and consider the image of $E_K(a,b,c)$ under $h$.  For any $x \in E_K(a,b,c)$,  $h(x)$ lies within $\lambda K +\epsilon$ of all three sides of $h(T)$ for some $T$ and hence within  $\lambda K +\epsilon+N$ of all three sides of $T'$.  Taking $K'=\max\{B_{D'}+ \delta_{D'}, \lambda K +\epsilon+N\}$, we conclude that the image of $E_K(a,b,c)$ under $h$ lies in 
$E_{K'}(a',b',c')$. 

Consider a 4-tuple $(a,b,c,d) \in X^{(4,D)}$ and let $(a',b',c',d') \in Y^{(4,D')}$ be its image under $\mbh$.  Choose a geodesic $\alpha \in (a,c)$  let $p$ and $q$ be the projections of $b$ and $d$ on $\alpha$.  Likewise, choose $\alpha' \in (a',c')$ and let $p',q'$ be the projections of $b'$ and $d'$ on $\alpha'$.  Then
up to a constant,  $|[a,b,c,d]| = d(p,q)$ and $|[a',b',c',d']|=d(p',q')$.  Hence to prove $\mbh$ is quasi-mobius, it suffices to bound $d(p',q')$ as a function of $d(p,q)$.

By part (1) of Lemma \ref{centers}, we have $p \in E_K(a,b,c)$ and $p' \in  E_{K'}(a',b',c')$.  By the discussion above, we also have $h(p) \in E_{K'}(a',b',c')$.  By part (2) of Lemma \ref{centers},   the diameter of $ E_{K'}(a',b',c')$ is bounded by a constant $C'$ depending only on $K', D'$, so $d(p',h(p)) < C'$. By the same argument, $d(q',h(q)) < C'$.  We conclude that
$$d(p',q') < d(h(p),h(q)) + 2C' \leq \lambda d(p,q) + \epsilon +2C'.$$
It follows that $\mbh$ is quasi-mobius with linear bounding functions with multiplicative constant $\lambda$.
\end{proof}

\section{Quasi-isometries induced by homeomorphisms}

The  goal of this section is to prove a converse of Theorem \ref{forward}.  
Namely, if $f: \mbX \to \mbY$ is a homeomorphism such that $f$ and $f^{-1}$ are both 2-stable and quasi-mobius, then $f$ is induced by a quasi-isometry $h: X \to Y$.

 \emph{We assume from now on that $\mbX$ contains at least three points and that there exists a cocompact group action on $X$.  Similarly for $Y$.}

\subsection{Extending \texorpdfstring{$f$}{f} to the interior}
Given a homeomorphism $f:\mbX \to \mbY$, we need to extend it to a map $h: X \to Y$.
Let us outline the steps involved in defining such an $h$.

Choose $D$ such that $\mbX^{(3,D)}$ is non-empty.  We begin by defining a map  
$$ \pi^D_X: \mbX^{(3,D)} \to X.$$  Fix  $K = B_D + \delta_D$.
Recall from Lemma \ref{centers} that $E_K(a,b,c)$ is bounded and hence has a well-defined barycenter. Define 
$$  \pi^D_X (a,b,c) = \textrm{the barycenter of $E_K(a,b,c)$.} $$   
When $D$ is fixed, we will generally omit it from the notation and denote the map by $\pi_X$.  

If $G$ is a group acting cocompactly by isometries  on $X$, then the induced action of $G$ on $\mbX$ preserves the contracting constants of bi-inifinte geodesics and $\pi_X$ is equivariant with respect to the induced action.  Choose a basepoint $x_0$ that lies in the image of $\pi_X$.  Since the action of $G$ on $X$ is cocompact, there is a ball $B(x_0,R)$ whose $G$-translates cover $X$, so every point in $X$ lies within $R$ of $\pi_X(a,b,c)$ for some $(a,b,c)$.

Now assume that $f : \mbX \to \mbY$ is a 2-stable homeomorphism and say
$f (\mbX^{(2,D)}) \subseteq \mbY^{(2,D')}$. As observed above, it follows that  $f (\mbX^{(n,D)}) \subseteq \mbY^{(n,D')}$ for all $n \geq 2$.  Set $\pi_X=\pi^D_X$  and let 
$\pi_Y = \pi^{D'}_Y : \mbY^{(D',3)} \to Y$ be the analogous map for $Y$.  We would like to define $h(x)$ to be the barycenter of the set
$$\Pi(x)= \pi_Y\circ f \circ \pi_X^{-1}(B(x,R)) \subset Y.$$ 
To do this, we must first prove that for any $x$,  $\Pi(x)$ has bounded diameter.

\begin{lemma} \label{flips}   There exists a constant $C_1$ depending only on $D$ such that for any  
$(a,b,c,d) \in \mbX^{(4,D)}$, the absolute value of one of the three cross-ratios $[a,b,c,d]$, $[a,c,b,d]$, 
$[a,c,d,b]$ is less than $C_1$.
\end{lemma}

\begin{proof}  Let $p=\pi_\alpha(b)$  and $q=\pi_\alpha(d)$.  Interchanging $b$ and $d$ if necessary (which changes only the sign of the cross-ratio), we may assume that $p$ lies between $a$ and $q$ as in Figure \ref{fig:flips}.

Consider the triangles $(a,b,p)$ and $(c,d,q)$.  We claim that the distance between these triangles is at least $d(p,q)-2B_D -4\delta_D$.    To see this, first note that 
the projections of $(p,b)$ and $(q,d)$ on $\alpha$ lie within $B_D$ of $p$ and $q$ respectively, so these rays remain at least $d(p,q) - 2B_D$ apart.  By the slim triangle property, the rays $(p,b)$ and $(q,c)$ are distance at least $d(p,q) - 2\delta_D$ apart, and likewise for $(q,d)$ and $(p,a)$.  Thus, $(p,a) \cup (p,b)$ and $(q,c)\cup (q,d)$ are separated by a distance of at least 
$d(p,q)-2B_D- 2\delta_D$.  Finally, for any geodesics $\gamma \in (a,b)$ and $\rho \in (c,d)$,  the slim triangle property then implies that $d(\gamma, \rho) > d(p,q)-2B_D- 4\delta_D$ as claimed. 

Now set $C = 3B_D+4\delta_D$ and suppose  $d(p,q) > C$.
Then by the discussion above, $d(\gamma, \rho) > B_D$.  Applying the Bounded Geodesic Image property we conclude that $\pi_\gamma (\rho)$ has diameter less than $B_D$. The projections of $c$ and $d$ on $\gamma$ lie within $B_D$ of $\pi_\gamma(\rho)$, thus 
$$d(\pi_\gamma(c),\pi_\gamma(d)) < 3B_D<C.$$
By Lemma \ref{supremum}, replacing $C$ by $C_1=C +6\delta_D$, we conclude that $|[a,b,c,d]| > C_1$   implies $|[a,c,b,d]| < C_1$.
\end{proof}

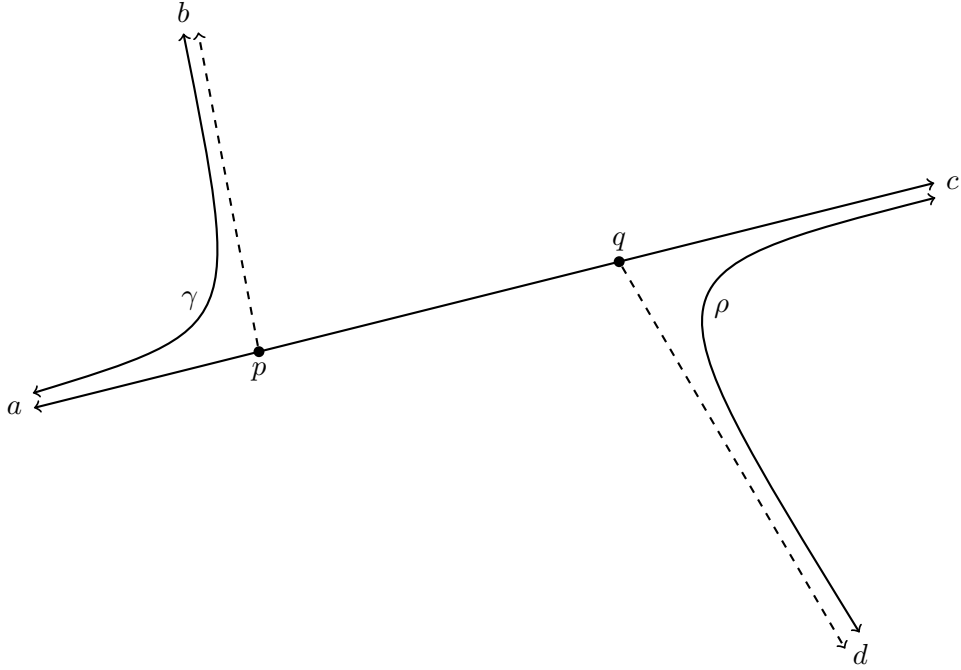
\begin{figure}
\begin{tikzpicture}[thick]

	\node[scale=0.1](a) at (0,0) {};
	\node[anchor = east] at (a) {$a$};
	\node[scale=0.1](b) at (2,5) {};
	\node[anchor = south] at (b) {$b$};
	\node[scale=0.1](c) at (12,3) {};
	\node[anchor = west] at (c) {$c$};
	\node[scale=0.1](d) at (11,-3) {};
	\node[anchor = north] at (d) {$d$};
	\node (up) at (0,0.2) {};
	\node (down) at (0,-0.2) {};
	\node (right) at (0.2,0) {};
	\node (left) at (-0.2,0) {};
	\draw[<->] (a) -- (c) node[pos=0.25,fill,scale=0.4,circle] (p){} node[pos = 0.25, anchor=north]{$p$} node[pos=0.65,fill,scale=0.4,circle] (q){} node[pos=0.65,anchor = south]{$q$};
	\draw[<->] ($(a)+(up)$)  ..controls ($(p)+1.5*(up)+1.1*(left)$) .. (b) node[pos=0.5,anchor = east]{$\gamma$};
	\draw[dashed,->] (p) -- ($(b)+(right)$);
	\draw[dashed,->] (q) -- ($(d)+(down)+(left)$);
	\draw[<->] (d) ..controls ($(q)+0.8*(down)+1.3*(right)$) .. ($(c)+(down)$) node [pos = 0.5,anchor = west]{$\rho$};

\end{tikzpicture}
\caption{Some cross-ratio is small}
\label{fig:flips}
\end{figure}

Let us now recall Lemma \ref{centers}. By part (3) of that lemma, we see that if $|[a,b,c,d]|$ is small, then the barycenters of the triangles $(a,b,c)$ and $(a,c,d)$ are close.  Thus, the lemma above says that starting with any $D$-triangle $T=(b,c,d)$ and any point $a$ with $(a,b,c,d) \in \mbX^{(4,D)}$, one of the three vertices of $T$ can be replaced by $a$ causing only a small change in the barycenter.  This is the key idea behind the proof of the next proposition.

\begin{proposition} \label{bndd expansion1}  Let $f : \mbX \to \mbY$ be a 2-stable, quasi-mobius homeomorphism.  Then for any $L \geq 0$, there exists a constant $C_2$ such that for any two D-triangles $(a,b,c), (u,v,w)$ in $\mbX^{(3,D)}$, 
$$d_X(\pi_X(a,b,c), \pi_X(u,v,w)) \leq L \Longrightarrow\\  d_Y(\pi_Y(f(a),f(b),f(c)), \pi_Y(f(u),f(v),f(w)) \leq C_2.$$
\end{proposition}

\begin{proof} Let $T_1=(a,b,c)$ and $T_2=(u,v,w)$ be two $D$-triangles and set $x=\pi_X(a,b,c)$ and $y= \pi_X(u,v,w)$.  Fix a constant $L \geq 0$.  We first show that there exists $\tilde D$, depending only on $D$ and $L$, such that if $d(x,y) \leq L$, then any geodesic from a vertex of $T_1$ to a vertex of $T_2$ is $\tilde D$-contracting.  Say, for example, that $\gamma \in (a,u)$ is such a geodesic.  By Lemma \ref{centers}, $x$ (resp. $y$) is uniformly bounded distance, say distance $C$, from the sides of any triangle representing $T_1$ (resp. $T_2$).  It follows that
$(a,x)$ is in the $C$-neighborhood of any geodesic in $(a,b)$, and $(u,y)$ is in the $C$-neighborhood of any geodesic in $(u,v)$.
Assuming $d(x,y) \leq L$ this also implies that $(u,x)$ is in the $(C+L)$-neighborhood of $(u,v)$.  It follows that there is a constant $D_1$, depending only on $C$ and $L$, such that $(a,x)$ and $(u,x)$ are $D_1$-contracting, and hence by the Contracting Triangles Property, a constant $\tilde D$ such that $[a,u]$ is $\tilde D$-contracting.  The same argument applied to other pairs of vertices shows that $(a,b,c,u,v,w) \in \mbX^{(\tilde D,6)}$.  

By Lemma \ref{flips}, there is a constant $C_1$ such that for any $(p,q,r,s) \in \mbX^{(\tilde D,4)}$, some permutation of $(p,q,r,s)$ has the absolute value of its cross-ratio bounded by $C_1$.  We will say that such a cross-ratio is ``small".   It follows from Lemma \ref{centers}(3), that if $[p,q,r,s]$ is small, then the barycenters  $\pi_X(p,q,r)$ and $\pi_X(p,s,r)$ are uniformly close, say at distance $< C'$.  In this case, we call the move from $(p,q,r)$ to $(p,s,r)$ a ``small flip".  

To prove the lemma, we begin by showing that applying at most 2 small flips to each of the triangles $T_1=(a,b,c)$ and $T_2=(u,v,w)$, we obtain a pair of triangles that share an edge.  To see this, first replace a vertex of $(a,b,c)$ by $w$.  Permuting $a,b,c$ if necessary, we get a small flip from $(a,b,c)$ to $(a,b,w)$.  Next, replace a vertex of $(a,b,w)$ by $v$.   If the flip to either $(v,b,w)$ or $(a,v,w)$ is small we are done since these share an edge with $(u,v,w)$.  

Suppose only the flip to $(a,b,v)$ is small.
In this case, consider the flips of $(u,v,w)$ obtained by replacing a vertex by $a$.  All of the resulting triangles $(a,v,w)$, $(u,a,w)$, $(u,v,a)$ all share an edge with either $(a,b,w)$ or $(a,b,v)$, so whichever flip is small, gives the desired pair of adjacent triangles. 

Since the barycenters $x,y$ of $T_1,T_2$ are at distance at most $L$, the centers of the resulting pair of adjacent triangles are at distance at most $L'=L+3C'$.  
In summary, there is a sequence of at most 5 triangles with vertices in $\{a,b,c,u,v,w\}$, beginning with $T_1$ and ending with $T_2$ such that consecutive triangles share an  edge and have centers at distance at most $L'$.  

Now apply $f$ to this sequence of triangles. Since $f$ is 2-stable, $f(a,b,c,u,v,w)$ lies in $\mbY^{(E,6)}$ for some $E \geq 0$.   
Using Lemma \ref{centers} and the quasi-mobius function $\Psi_{\tilde D}$, one gets a bound on the distance between the centers of consecutive triangles, and hence a bound on the distance between the centers of $f(T_1)$ and $f(T_2)$.  This proves the proposition.
\end{proof}

In particular, it follows that for any $x \in X$, the set $\pi_Y(f(\pi_X^{-1}(B(x,R)))) \subset Y$ has bounded diameter with the bound, say $M$, depending only on the choice of  $D, D'$ and $R$.  Set
$$\Pi (x) = \pi_Y(f(\pi_X^{-1}(B(x,R)))) $$
and define a map $h : X \to Y$ by
$$ h(x) = \textrm{ the barycenter of $\Pi (x)$}. $$ 
We call $h$ an extension of $f$ to $X$.
While the definition of $\Pi$, and hence $h$, depends on a choice of $D, D'$ and $R$, increasing any of these constants just increases the size of $\Pi$, and hence the resulting map differs from the original by a uniformly bounded amount.

\begin{proposition} \label{bndd expansion2} For any $L, D \geq 0$, there exists $C_3$ such that 
$$d_X(x,y) \leq L \Longrightarrow  d_Y(h(x), h(y)) \leq C_3.$$
\end{proposition}  

\begin{proof}  Let $(a,b,c)$ and $(u,v,w)$ be $D$-triangles with  $\pi_X(a,b,c)$ and $\pi_X(u,v,w)$ in $B(x,R)$ and $B(y,R)$ respectively.  
If $d_X(x,y) \leq L$, then $d_X(\pi_X(a,b,c), \pi_X(u,v,w)) \leq L+2R$.  So by Lemma \ref{bndd expansion1}, there exists $C$ 
depending only on $D, L, R$ such that  
$$d_Y(\pi_Y(f(a),f(b),f(c)), \pi_Y(f(u),f(v),f(w)) \leq C.$$
Now $\pi_Y(f(a),f(b),f(c))$ is an element of $\Pi(x)$ which is a set of diameter at most $M$, so its center, $h(x),$ lies within $M$ of  $\pi_Y(f(a),f(b),f(c))$.  Similarly, $h(y)$ lies within $M$ of  $\pi_Y(f(u),f(v),f(w))$.
Thus, $d_Y(h(x), h(y)) \leq C + 2M$.  
\end{proof}

\subsection{Main theorem}

\begin{theorem}\label{reverse} Let $X$ and $Y$ be CAT(0) spaces with proper, cocompact actions $G \curvearrowright X$ and $H  \curvearrowright Y$, and assume $\mbX$ has at least 3 points.
Suppose $f : \mbX \to \mbY$ is a 2-stable, quasi-mobius homeomorphism, and likewise for $f^{-1}$.  Then there exists a quasi-isometry  $h: X \to Y$  with $\partial_*h=f$.
\end{theorem}

\begin{proof}  Choose $D'$ so that $f(\mbX^{(2,D)}) \subseteq \mbY^{(2,D')}$ and $f^{-1}(\mbY^{(2,D)}) \subseteq \mbX^{(2,D')}$.  Choose
 $R>0$ so that both $X$ and $Y$ are covered by the $R$-neighborhood of an orbit.   Using these constants, define $h$ to be the extension of $f$ to $X$ and $h^{-1}$ to be the extension of $f^{-1}$ to $Y$.  
That is,  $h(x) =$ barycenter of $\Pi(x)$ and $h^{-1} =$ barycenter of $\Pi(y)$, where
$$\Pi(x)= \pi_Y(f (\pi_X^{-1}(B(x,R)))) $$ 
$$\Pi(y)= \pi_X(f^{-1}(\pi_Y^{-1}(B(y,R))))$$ 
As above, let $M$ denote an upper bound on the diameter of $\Pi(x)$.

To prove that $h$ is a quasi-isometry, it suffices to show 
\begin{itemize}
\item[(i)] for all $x,y \in X$ and $p,q \in Y$ there are linear bounds $d_Y(h(x),h(y)) \leq A\, d_X(x,y) + B$
and $d_Y(h^{-1}(p),h^{-1}(q)) \leq A' d_Y(p,q) + B'$, and
\item[(ii)]  $h$ and $h^{-1}$ are quasi-inverses.
\end{itemize}

For (i), let $S$ be a finite generating set for $G$.  Choose a base point $x_0 \in X$.
Approximate the geodesic from $x$ to $y$ by a sequence of orbit points $g_0x_0, g_1x_0, \dots g_nx_0$ such that $g_{i+1}=g_is_i$ for some generator $s_i \in S$.  Now map this sequence by $h$ into $Y$.  Since the distance between consecutive points is bounded, Proposition \ref{bndd expansion2} implies that there exists $C_3$ such that $d_Y(h(g_0x_0),h(g_nx_0)) \leq C_3\, n$
and hence $d_Y(h(x),h(y)) \leq C_3\,n + 2M$.  Since the inclusion of $G$ into $X$ as the orbit of $x_0$ is a quasi-isometry, there exists $\lambda, \epsilon$ such that $n=d_G(g_0,g_n) \leq \lambda d_X(x,y) + \epsilon$, so 
$$d_Y(h(x),h(y)) \leq C_3\lambda\, d_X(x,y) +(C_3 \epsilon + 2M).$$
An analogous argument for $f^{-1}$ gives an upper bound on $d_X(h^{-1}(p),h^{-1}(q))$ as a linear function of $d_Y(p,q)$.

Next we prove that $h$ and $h^{-1}$ are quasi-inverses.  Say $D''$ is such that $f^{-1}(\mbY^{(2,D')}) \subseteq \mbX^{(2,D'')}$. Choose $R' \geq \max\{R,M\}$ and note that $\Pi(x) \subset B(h(x), R')$ 
for all $x$.  Let  $\pi'_X= \pi^{D''}_X$ and 
define  $\hat h^{-1}(y)$ to be the barycenter of the set  
$$\Pi'(y)= \pi'_X(f^{-1}(\pi_Y^{-1}(B(y,R')))). $$ 
Since $\Pi'$ is obtained from $\Pi$ simply by increasing the constants, $\Pi(y) \subseteq \Pi'(y)$ for all $y$, so the distance between $h^{-1}$ and $\hat h^{-1}$ is uniformly bounded.

Now say $y=h(x)$.   Then for any $D$-triangle $(a,b,c)$ whose image $z=\pi_X(a,b,c)$ lies in $B(x,R)$, 
we have $\pi_Y(f(a,b,c)) \in \Pi(x) \subset B(y,R')$, and hence $z$ lies in $\Pi'(y)$.  
This set has bounded diameter, say $M',$ independent of $y$ and its barycenter is $\hat h^{-1}(y)$, so
$$d(x,\hat h^{-1}h(x)) = d(x,\hat h^{-1}(y)) \leq d(x,z) + d(z, \hat h^{-1}(y)) \leq R + M'.$$
It follows that  $h^{-1}\circ h$ is also at bounded distance from the identity map. An analogous argument proves that the same holds for $h\circ h^{-1}$.

It remains to show that $\mbh =f$.  For this, we will use the fact  that one can put a topology on $X \cup \mbX$ with the property that for a fixed basepoint $x$, a sequence $\{x_i\}$ of points in $X \cup \mbX$ converges to a point $p$ on $\mbX$ if and only if the geodesics from $x$ to $x_i$ have bounded contracting constants and converge pointwise to the geodesic from $x$ to $p$. (See \cite{Mu} for details.)

Choose a basepoint $x \in X$ such that  $x = \pi_X(a,b,c)$ is the barycenter of some $D$-triangle $(a,b,c) \in \mbX^{(D,3)}$.  Let $p$ be a point in $\mbX$ and say the ray $\alpha$ from $x$ to $p$ is $D'$-contracting for some $D' \geq D$.   Approximate $\alpha$ by a sequence of orbit points $\{g_ix\}$  converging to $p$.  Set $x_i=g_ix$, so $x_i$ is the barycenter of the triangle $(a_i, b_i, c_i) = (g_ia, g_ib, g_ic)$.  Note that there exists $D''$ such that the geodesic ray from $x$ to any vertex in any one of these triangles is $D''$-contracting.  This follows from the Contracting Triangle Property since the triangle $(x,x_i,a_i)$ has two sides, $(x,x_i)$ and $(x_i,a_i)$, which stay close to $D'$-contracting rays, and similarly for $b_i$ and $c_i$.  

Consider the sequences $\{a_i\}, \{b_i\}$, and  $\{c_i\}$.  We claim that at least two of these sequences converge to $p$.  Suppose not.  Say $\{a_i\}$ and  $\{b_i\}$ do not converge to $p$.
Then $\{a_i\}$ and  $\{b_i\}$ are disjoint from some open neighborhood $U$ of $p$.   It follows that for all $i$, the projections of $a_i$ and $b_i$ on $\alpha$ lie in some bounded segment $[\alpha(0), \alpha (R)]$. Since $x_i$ lies bounded distance from some geodesic in $(a_i,b_i)$, the projections of every $x_i$ on $\alpha$ also lies in a bounded segment.  But this contradicts the assumption that $\{x_i\}$ converges to $p$.

Since $f: \mbX \to \mbY$ is a homeomorphism, it follows that two of the sequences  $\{f(a_i)\}, \{f(b_i)\},\{f(c_i)\}$ converge to $f(p)$ and hence the barycenters $\{y_i\}$ of these triangles also converge to $f(p)$.  Since $h(x_i)$ lies bounded distance from $y_i$, we conclude that 
$\partial_* h(p) = \lim_{i \to \infty} h(x_i) = f(p)$.
\end{proof}

Combining Theorems \ref{forward} and \ref{reverse} we thus have
\begin{theorem} Let $X$ and $Y$ be proper, cocompact CAT(0) spaces with at least 3 points in their Morse boundaries.   A homeomorphism $f : \mbX \to \mbY$ is induced by a quasi-isometry  $h : X \to Y$ if and only if $f$ is 2-stable and quasi-mobius.  
\end{theorem}

We conclude by noting that the homeomorphisms described in Example \ref{not 2-stable} that are not 2-stable, also fail to be quasi-mobius.  For example, setting $m=n+1$, the cross-ratio
$$|[r_{n,0},r_{m,0},r_{n,1},r_{m,1}]|=1$$
whereas after applying $f$ we get a cross-ratio of
$$|[r_{-n,0},r_{-m,0},r_{n,1},r_{m,1}]| > 2n-1$$

It seems reasonable to conjecture that this is always the case.
\begin{conjecture}  If $X$ and $Y$ are proper, cocompact CAT(0) spaces, then every quasi-mobius homeomorphism, $f: \mbX \to \mbY,$ is 2-stable and hence induced by a quasi-isometry.
\end{conjecture}

\bibliographystyle{utphys}
\bibliography{quasi-mobius}

\end{document}